\documentclass[preprint,12pt]{elsarticle}
\usepackage{amssymb}
\usepackage{amsthm}
\usepackage{amsmath}
\usepackage{graphics}
\newtheorem{theorem}{Theorem}[section]
\newtheorem{lemma}[theorem]{Lemma}

\begin{document}
\begin{frontmatter}
\title{Numerical solution for a general class of nonlocal nonlinear wave equations}

 \author[ikbu]{Handan Borluk }
 \author[itu]{Gulcin M. Muslu\corref{cor1}}
 \ead{gulcin@itu.edu.tr}
 \cortext[cor1]{Corresponding author}
 \address[ikbu]{Istanbul Kemerburgaz University, Department of Basic Sciences, Bagcilar 34217,
         Istanbul,  Turkey.}
 \address[itu]{Istanbul Technical University, Department of Mathematics, Maslak 34469,
         Istanbul,  Turkey.}

\begin{abstract}
A class of nonlocal nonlinear wave equation arises from the modeling of a one dimensional motion in a nonlinearly, nonlocally elastic medium. The equation involves a kernel function with nonnegative Fourier transform. We discretize the equation by using Fourier spectral method in space and we prove the convergence of the semidiscrete scheme. We then use a fully-discrete scheme, that couples Fourier pseudo-spectral method in space and 4th order Runge-Kutta in time, to observe the effect of the kernel function on solutions.
To generate solitary wave solutions numerically,  we use the Petviashvili's iteration method.
\end{abstract}
\end{frontmatter}

\renewcommand{\theequation}{\arabic{section}.\arabic{equation}}
\setcounter{equation}{0}

\section{Introduction}
In this article, we study the nonlocal nonlinear wave equation
\begin{equation}
u_{tt}= \beta\ast(u+g(u))_{xx},  \label{nl}
\end{equation}
where $g$ is a nonlinear function of $u$. Here $\ast$ denotes the convolution in spatial domain and the kernel function $\beta(x)$
is an integrable function. Equation \eqref{nl} has been derived to model one-dimensional motion in an
infinite medium with nonlinear and nonlocal elastic properties from nonlocal elasticity theory \cite{duruk1, duruk2}.
 Duruk et. al. have proved the global existence of the Cauchy problem corresponding to \eqref{nl}
together with initial conditions
\begin{equation}
u(x,0)=\phi(x),~~~~u_t(x,0)=\psi(x)\label{ic}
\end{equation}
with enough smoothness on the initial data, and continuity condition on the nonlinear function $g$ in \cite{duruk1}. They assume that the Fourier
transform of the kernel function $\beta(x)$ satisfies
\begin{equation}
0\leq \widehat{\beta}(k)\leq C (1+k^2)^{-r/2},~~\mbox{for~all~}k\in\mathbb{R}\label{condbeta}
\end{equation}
for some constant $C>0$, $r\in \mathbb{R}$ and $r\geq 2$.
Suppose that $u$ and all its derivatives converge to zero sufficiently and
rapidly as $x\rightarrow \pm \infty$. For the solutions  of the eq. \eqref{nl}  subjected to
these boundary conditions the conserved quantity (energy) \cite{duruk1} is given by
\begin{equation*}
    E\left( t\right) =\left\Vert Pu_t \right\Vert^{2}+\left\Vert u \right\Vert^{2}+2   \int_{\Bbb R}  G\left( u\right) dx.
    \end{equation*}
\noindent
Here the operator $P$  is defined as
$$Pv=\mathcal{F}^{-1} \left( \frac{\hat{v}\left( k \right)}{\left\vert k \right\vert \sqrt{\widehat{\beta}(k)} } \right)$$
where $\mathcal{F}^{-1}$ is  the inverse Fourier transform  and $\displaystyle G(u)=\int_{0}^{u} g(s)ds$.

\noindent
For some suitable choice of the kernel function $\beta(x)$ the equation \eqref{nl} becomes well-known
nonlinear wave equations.
In the case of  the Dirac delta function,  $\beta =\delta$, the condition \eqref{condbeta} is satisfied by $r=0$ and the equation (\ref{nl}) is  the well-known
nonlinear wave equation
\begin{equation*}
    u_{tt}-u_{xx}=g\left( u\right) _{xx}.
\end{equation*}
Choosing the  exponential kernel \cite{eringen2}, $\beta(x)={\frac{1}{2}}e^{-|x|}$, the equation (\ref{nl}) becomes the improved Boussinesq equation (IBq),
\begin{equation*}
    u_{tt}-u_{xx}-u_{xxtt}=(g(u))_{xx}~.
\end{equation*}
Since $\widehat{\beta}(k )=(1+k^{2})^{-1}$, we have $r=2$. A final example is the double-exponential kernel \cite{lazar}
$$\beta (x)={\frac{1}{{2(c_{1}^{2}-c_{2}^{2})}}}(c_{1}e^{-|x|/c_{1}}-c_{2}e^{-|x|/c_{2}})$$
where $c_{1}$ and $c_{2}$ are real and positive constants.
As $\widehat{\beta}(k )=(1+\eta _{1}k ^{2}+\eta _{2}k ^{4})^{-1}$ where
$\eta _{1}=c_{1}^{2}+c_{2}^{2}$ and $\eta _{2}=c_{1}^{2}c_{2}^{2}$,  we have $r=4$.
In this case the equation (\ref{nl})  becomes the higher-order Boussinesq equation (HBq),
\begin{equation*}
    u_{tt}-u_{xx}-\eta_{1}u_{xxtt}+\eta_{2}u_{xxxxtt}=(g(u))_{xx}~.
\end{equation*}
For these special cases there are numerous studies in the literature both numerical and analytical (see \cite{borluk} and \cite{topkarci} and the references there in).  The existence and stability of solitary wave solutions of the nonlocal nonlinear wave equation \eqref{nl} has been shown in a recent study \cite{erkip}. The nonlocal term makes it harder to investigate the equation analytically.
Therefore the numerical solution is important to understand the evolution of solutions in time and the effect of the kernel on the solutions.

The existence of the convolution integral and the condition \eqref{condbeta} on the Fourier coefficients of the kernel function
$\beta(x)$ make the spectral methods a natural choice for the numerical solution. Spectral methods are widely used for local equations. Some of the  numerical studies using spectral methods for the nonlocal equations are as follows:
In \cite{pelloni2,pelloni}, the authors have studied a class of nonlinear nonlocal dispersive wave equation by using the combination of Fourier Galerkin spectral method and explicit leap-frog scheme. In \cite{kalisch}, the author has studied regularized Benjamin-Ono equation by using both Fourier Galerkin and Fourier collocation methods.

In the recent studies \cite{borluk,topkarci} a Fourier pseudo-spectral method has been used for the numerical solutions of the IBq and HBq equations.
Thus the efficiency of  the method has been tested for the special cases of this class of nonlocal equations. Therefore in this study we use Fourier pseudo-spectral method, together with a fourth order Runge-Kutta method,
to solve \eqref{nl} numerically. In Section 2, we give notations and preliminaries. In Section 3, we prove the convergence of the semi-discrete scheme. In Section 4, we introduce a fully discrete scheme combining the Fourier pseudo-spectral method and fourth order Runge-Kutta method for space and time discretization, respectively.
Although the exact forms of solitary wave solutions
are  known for the IBq and HBq equations,  we do not have the exact form of the solitary wave solutions for the general nonlocal
nonlinear wave equation. Therefore,  we use the Petviashvili's iteration  method  to construct the solitary wave profile numerically in Section 5.   Finally, we present some numerical implementation of the method in Section 6.

\section{Notations and preliminaries}
We use the function space $L^2(\Omega)$ and the Sobolev space $H^s$.  $(\cdot, \cdot)$ denotes the  $L^2$ inner product
and $\|\cdot\|$ is the corresponding norm defined, respectively as
\begin{equation}
(u,v)=\int_\Omega u(x)\overline{v(x)}dx, \hspace*{10pt}\mbox{and}\hspace*{10pt} \|u\|^2=(u,u)
\end{equation}
where $\Omega=(-L,L)$. The periodic Sobolev space $H^s_p(\Omega)$ is equipped with the norm
\begin{equation}
\|u\|_s^2= \sum_{k=-\infty}^{\infty} (1+k^{2})^s|\hat{u}_k|^2.
\end{equation}
The Banach space $X_s=C^1([0,T];H^s_p(\Omega) )$
is the space of all continuous functions in $H^s_p(\Omega)$ whose distributional derivative  is also in $H^s_p(\Omega)$, with norm
$~\left\Vert u \right\Vert_{X_s}^2=\displaystyle \max_{ t \in [0,T]}(\left\Vert u(t) \right\Vert_{s}^2+\left\Vert u_t (t) \right\Vert_{s}^2)$.
Here the Fourier coefficients of a function is given as
$$\hat{u}_k=\displaystyle\frac{1}{2L} \int_\Omega u(x)e^{{ik\pi x}/{L}} dx.$$
For a positive integer $N$ the space of trigonometric polynomials of degree $N/2$ is
\begin{equation}
S_N=\mbox{span}\{e^{{ik\pi x}/{L} } | -N/2\leq k\leq N/2-1 \}
\end{equation}
The operator $P_N:L^2(\Omega)\rightarrow S_N$ denotes the orthogonal projection defined as
\begin{equation}
P_Nu(x)=\sum_{k=-N/2}^{N/2-1} \hat{u}_k~ e^{{ik\pi x}/{L}}.
\end{equation}
So that the projection operator $P_N$ has  the property
\begin{equation}\label{orthogonal}
(u-P_Nu,\varphi)=0, \hspace*{20pt} \forall \varphi \in S_N.
\end{equation}
$P_N$ commutes with derivative in the distributional sense:
\begin{equation}\label{commute}
D_x^n P_Nu=P_ND_x^nu \hspace*{20pt} \mbox{and} \hspace*{20pt} D_t^n P_Nu=P_ND_t^n u.
\end{equation}
Here $D_x^n$ and $D_t^n$ stand
for the $n$th-order classical partial derivative with respect to $x$ and $t$, respectively.
In the rest of study, $C$ denotes a generic constant.

\setcounter{equation}{0}
\section{The semi-discrete scheme}
\setcounter{equation}{0}
In this section we present a Fourier pseudo-spectral method for the spatial discretization and
then we prove the  convergence of the semi-discrete scheme. The equation \eqref{nl} equivalently
can be written as
\begin{equation}
u_{tt}= {\cal H} [u+g(u)]_{xx}, \label{new-nl}
\end{equation}
where ${\cal H}$ is the convolution operator given by
$$ {\cal H}v(x)=\int_{\Omega}\beta(x-y)v(y)dy.$$
Then, the equation \eqref{new-nl}  becomes
\begin{equation}
{\cal H}^{-1}u_{tt}=  [u+g(u)]_{xx}. \label{new-nl2}
\end{equation}
We define  the operator ${\cal H}^{-1}$  as
${\cal H}^{-1}v= {\cal F}^{-1}({\hat{v}(k)}/{\hat{\beta}(k)})$
where  ${\cal F}^{-1}$ is the inverse Fourier transform. The semi-discrete Fourier pseudo-spectral scheme for (\ref{new-nl2}) with the initial conditions (\ref{ic}) are given as
\begin{eqnarray}
&& {\cal H}^{-1}u^N_{tt}=u^N_{xx}+P_N(g(u^N))_{xx}, \label{sd1}  \\
&& u^N(x,0)=P_N\phi(x), \hspace*{20pt} u^N_t(x,0)=P_N\psi(x) \label{sd2}
\end{eqnarray}
where $u^N(x,t)$ is a function from  $[0,T]$ to $S_N$.
\noindent We use the following lemmas for the convergence proof of the semi-discrete scheme:
\begin{lemma}\cite{canuto, rashid} \label{lemma}
For any real $0\leq \mu \leq s$, there exists a constant $C$ such that
\begin{equation}
\|u-P_N u\|_\mu \leq C N^{\mu-s} \|u\|_s,  \hspace*{20pt} \forall u\in H^s_p(\Omega).
\end{equation}
\end{lemma}
\begin{lemma}\cite{runst} \label{lemma2}
Assume that $f\in C^k(\mathbb{R})$, $u,v\in H^s(\Omega)\cap L^\infty (\Omega)$ and $k=[s]+1$, where $s\geq 0$. Then we have
\begin{equation}
\|f(u)-f(v)\|_s \leq C(M) \|u-v\|_s
\end{equation}
if $\|u\|_\infty \leq M, \|v\|_\infty \leq M, \|u\|_s \leq M$  and $\|v\|_s \leq M$, where $C(M)$ is a constant dependent
on $M$ and $s$.
\end{lemma}
\begin{lemma}\label{lemma3}
Let $r\geq 2$ and $v\in H^{r/2}(\Omega)$. If the kernel function $\beta(x)$ satisfies the condition \eqref{condbeta}, then there exists a constant $C$ such that
\begin{equation}
\|v\|_{r/2} \leq C \|{\cal H}^{-1/2}v\|.  \label{opesz}
\end{equation}
\end{lemma}
\begin{proof} The proof directly follows from the definition of the operator ${\cal H}^{-1/2}$ and  the condition \eqref{condbeta}.
\end{proof}

\noindent
The following theorem states our main result.
\begin{theorem} \label{theorem}
Let  $s\geq r/2$,  $g\in C^2(\mathbb{R})$ and $u(x,t)$ be the solution of the periodic initial value problem \eqref{nl}-\eqref{ic} satisfying
$u(x,t)\in C^1([0,T];H^s_p(\Omega) ) $ for any $T>0$ and  $u^N(x,t)$ be the solution of the semi-discrete scheme (\ref{sd1})-(\ref{sd2}).
There exists a constant $C$, independent of $N$, such that
\begin{equation}
\|u-u^N\|_{X_{r/2}} \leq C(T) N^{r/2-s}  \|u\|_{X_s}
\end{equation}
for the initial data $\phi,\psi \in H^s_p(\Omega)$.
\end{theorem}
%%%%%%%%%%%%%%%%%%%%%%%%%%%%%
\begin{proof}
To estimate the term $\displaystyle \|u-u^N\|_{X_{r/2}} $ we first use the triangle inequality to write
\begin{equation} \label{u-triangle}
\|u-u^N\|_{X_{r/2}}\leq\|u-P_Nu\|_{X_{r/2}}+\|P_Nu-u^N\|_{X_{r/2}}.
\end{equation}
For the first term on the right-hand side, we use Lemma \ref{lemma} and the fact that the projection operator
$P_N$ commutes with the differentiation, to obtain the following estimate
\begin{equation}
\|(u-P_Nu)_t(t)\|_{r/2}\leq C N^{r/2-s}\|u_t(t)\|_{s}, \label{first}
\end{equation}
for  $s\geq r/2$. Thus, we bound the first term in (\ref{u-triangle}) as
\begin{equation} \label{estimation1}
\|u-P_Nu\|_{X_{r/2}}\leq C N^{r/2-s}\|u\|_{X_s}.
\end{equation}
To estimate the second term  $\|P_Nu-u^N\|_{X_{r/2}}$ at the right-hand side of the inequality \eqref{u-triangle},
we subtract the equation \eqref{sd1} from \eqref{new-nl}  and take the inner product with $\varphi\in S_N$;
\begin{equation}\label{inner}
\left({\cal H}^{-1}(u-u^N)_{tt}- (u-u^N)_{xx}-(g(u)-P_Ng(u^N))_{xx},~\varphi\right)=0.
\end{equation}
Orthogonality of the projection operator $P_N$ gives
\begin{equation}
\left({\cal H}^{-1}(u-P_Nu)_{tt},~\varphi\right)=0
\end{equation}
for all  $\varphi\in S_N$, using this result together with \eqref{orthogonal}   the equation \eqref{inner} turns to
\begin{equation} \label{simpinner}
\left({\cal H}^{-1}(P_Nu-u^N)_{tt}- (P_Nu-u^N)_{xx}-(g(u)-P_Ng(u^N))_{xx},~\varphi\right)=0.
\end{equation}
We now set  $\varphi=(P_Nu-u^N)_{t}$  in \eqref{simpinner}. Since ${\cal H}^{-1/2}$ is a self-adjoint operator, we have
\begin{equation}
\left({\cal H}^{-1}(P_Nu-u^N)_{tt},~(P_Nu-u^N)_{t}\right)=\frac{1}{2}\frac{d}{dt}\|{\cal H}^{-1/2}(P_Nu-u^N)_{t}(t)\|^2. \label{norm1}
\end{equation}
Spatial periodicity and the integration by parts yield
\begin{equation}
\left((P_Nu-u^N)_{xx},~(P_Nu-u^N)_{t}\right)=-\frac{1}{2}\frac{d}{dt}\|(P_Nu-u^N)_{x}(t)\|^2. \label{norm2}
\end{equation}

\noindent

For the nonhomogeneous term in \eqref{simpinner} we use the Cauchy-Schwarz \mbox{inequality}, triangle inequality and  Lemma \ref{lemma2}. Therefore we have
\begin{eqnarray}\label{right-hand}
&& \hspace*{-50pt}\left( (g(u)-g(u^N))_{xx},(P_Nu-u^N)_{t}\right)  \notag \\
\hspace*{30pt} && \leq \left|\left( (g(u)-g(u^N))_{x},~(P_Nu-u^N)_{xt}\right)\right| \notag \\
\hspace*{30pt} && \leq \|(g(u)-g(u^N))_{x}\|~\|(P_Nu-u^N)_{xt}(t)\|\notag \\
%\hspace*{30pt} && \leq \frac{1}{2}\left( \|(g(u)-g(u^N))_{x}\|^2+\|(P_Nu-u^N)_{xt}(t)\|^2\right)\notag \\
\hspace*{30pt} && \leq \frac{1}{2}\left( \|g(u)-g(u^N)\|_1^2+\|(P_Nu-u^N)_{t}(t)\|_1^2\right)\notag \\
\hspace*{30pt} && \leq C \left( \|(u-u^N)(t)\|_1^2+\|(P_Nu-u^N)_{t}(t)\|_1^2\right)\notag\\
\hspace*{30pt} && \leq C \left( \|(u-P_Nu)(t)\|_{r/2}^2+\|(P_Nu-u^N)(t)\|_{r/2}^2+ \|(P_Nu-u^N)_{t}(t)\|_1^2\right) \notag \\ \label{nonhom}
\end{eqnarray}
for $r\geq 2$. Combining the results of \eqref{norm1}-\eqref{nonhom} the equation \eqref{simpinner} becomes
\begin{eqnarray} \label{inner3}
&&\hspace{-15pt}\frac{1}{2}\frac{d}{dt}\left( \|{\cal H}^{-1/2}(P_Nu-u^N)_{t}(t)\|^2 + \|(P_Nu-u^N)_{x}(t)\|^2 \right )\notag\\
&&\hspace{15pt}\leq C \left( \|(u-P^Nu)(t)\|_{r/2}^2+\|(P^Nu-u)(t)\|_{r/2}^2+\|(P_Nu-u^N)_{t}(t)\|_1^2\right). \notag \\
\end{eqnarray}
We now add the terms
$$ \left( P_Nu-u^N, (P_Nu-u^N)_{t} \right )+ \sum_{i=2}^{r/2}\left( D^i_x(P_Nu-u^N), D^i_x(P_Nu-u^N)_{t} \right ) $$
to both sides of the inequality \eqref{inner3} and use that
$$\left( D^i_x(P_Nu-u^N), D^i_x(P_Nu-u^N)_{t} \right )=\frac{1}{2}\frac{d}{dt}\| D^i_x(P_Nu-u^N)(t) \|^2, $$
and
$$\left( D^i_x(P_Nu-u^N), D^i_x(P_Nu-u^N)_{t} \right )
\leq C \left ( \| D^i_x(P_Nu-u^N)(t) \|^2 + \| D^i_x(P_Nu-u^N)_{t} (t)\|^2\right ).$$
Using all these results and Lemma \ref{lemma3}, the inequality \eqref{inner3}  becomes
\begin{eqnarray}
&&\frac{1}{2}\frac{d}{dt}\left( \|{\cal H}^{-1/2}(P_Nu-u^N)_{t}(t)\|^2 + \|(P_Nu-u^N)(t)\|_{r/2}^2 \right ) \notag \\
&& \hspace*{40pt}\leq C \left ( \|(u-P_Nu)(t) \|_{r/2}^2 + \|(P_Nu-u^N)(t) \|_{r/2}^2
+ \| (P_Nu-u^N)_{t} (t)\|_{r/2}^2\right ) \notag \\
&& \hspace*{40pt}\leq C \left ( \|(u-P_Nu)(t) \|_{r/2}^2 + \|(P_Nu-u^N)(t) \|_{r/2}^2
+  \|{\cal H}^{-1/2}(P_Nu-u^N)_{t}(t)\|^2  \right ). \notag
\end{eqnarray}
Note that  $ \|{\cal H}^{-1/2}(P_Nu-u^N)_{t}(0) \|^2 = \| (P_Nu-u^N) (0)\|^2=0 $. Using Gronwall Lemma followed by Lemma \ref{lemma}  we conclude that
\begin{eqnarray}
&&\hspace*{-90pt} \|{\cal H}^{-1/2}(P_Nu-u^N)_t(t)\|^2 + \| (P_Nu-u^N) (t) \|_{r/2}^2 \notag \\
&& \hspace*{20pt} \leq \int_0^t \| u(\tau)-P_Nu(\tau) \|_{r/2}^2~ e^{C(t-\tau)} d\tau\notag  \\
                       && \hspace*{20pt} \leq  C(T) ~ N^{r-2s} \int_0^t \| u(\tau)\|_s^2  ~d\tau
                         \label{esz}\end{eqnarray}
for $s\geq r/2$. Thanks to Lemma \ref{lemma3}, we know that
\begin{eqnarray}
&&\hspace*{-50pt} \|(P_Nu-u^N)_{t}(t)\|_{r/2}^2 + \| (P_Nu-u^N) (t) \|_{r/2}^2  \notag \\
&& \hspace*{30pt} \leq \|{\cal H}^{-1/2}(P_Nu-u^N)_{t}(t)\|^2 + \| (P_Nu-u^N)_t (t) \|_{r/2}^2.  \notag
\end{eqnarray}
Finally, we have the estimate
\begin{equation}
\| P_Nu-u^N \|_{X_{r/2}} \leq C(T)  N^{r/2-s}~ \| u\|_{X_s}. \label{estimation2}
\end{equation}

\noindent
The inequalities \eqref{estimation1} and \eqref{estimation2}  give the estimation for \eqref{u-triangle}, which  \mbox{completes} the proof of Theorem \ref{theorem}.
\end{proof}

\noindent
Note that the result of the Theorem \ref{theorem} coincides with the special cases of the  equation \eqref{nl},
the IBq  and HBq equations, for which $r=2$ and $r=4$, respectively \cite{borluk,topkarci}.
\setcounter{equation}{0}
\section{The fully-discrete scheme}

\vspace*{-5pt}
We solve the  equation \eqref{nl} by combining a Fourier pseudo-spectral method for the space component
and a fourth-order Runge Kutta scheme (RK4) for time. If the spatial period $[-L,\, L]$ is
normalized to $[0,2\pi]$ using the transformation
\mbox{$X=\pi(x+L)/L$}, the equation \eqref{nl} with $g(u)=u^p$ becomes
\begin{equation}\label{nonlocal-zero-2pi}
u^N_{tt}= \left(\frac{\pi}{L}\right)^2(\beta \ast (u^N+(u^N)^p))_{XX}.
\end{equation}
The space interval $[0,2\pi]$ is divided into $N$ equal subintervals with grid spacing
$\Delta X=2\pi/N$, where the integer $N$ is even. The spatial grid points are then given by
$X_{j}=2\pi j/N$,  $j=0,1,2,...,N$. The time interval $[0,T]$ is divided into $M$ equal subintervals with grid spacing
$\Delta t=T/M$.  The approximate solution to
$u^N(X_{j},t)$ is denoted by $U_{j}(t)$. Applying the discrete Fourier transform to the equation \eqref{nonlocal-zero-2pi}
we have
\begin{equation}
(\tilde{U}_k)_{tt}= -\left(\frac{\pi k}{L}\right)^2\,\tilde{\beta}_k\,[\,\tilde{U}_k+(\tilde{U}^p)_k\,]. \label{fourier}
\end{equation}
The  above equation  can be written as an ordinary differential equation system
\begin{eqnarray}
&& (\tilde{U}_k)_t=\tilde{V}_k \label{sis1} \\
&& (\tilde{V}_k)_t= -(\pi k/L)^2\,\tilde{\beta}_k\,[\,\tilde{U}_k+(\tilde{U}^p)_k\,].\label{sis2}
\end{eqnarray}
In order to handle the nonlinear term we use a pseudo-spectral approximation.
We use the fourth order Runge-Kutta method to solve the resulting ODE system
\eqref{sis1}-\eqref{sis2} in time. Finally, we find the approximate solution by using the inverse discrete Fourier transform.

\setcounter{equation}{0}
\section{The Petviashvili's iteration method}
One of the most effective numerical methods to generate solitary wave solution was first  proposed by V. I. Petviashvili  for the  Kadomtsev-Petviashvili equation
in \cite{petviashvili}. Since we do not have the exact form of the solitary wave solutions for the general nonlocal
nonlinear wave equation,  we use the Petviashvili's iteration  method \cite{pelinovsky, petviashvili, yang} to construct the solitary wave solution numerically.

\par
The solitary wave solutions of \eqref{nl} are of the form $u(x,t)=\phi(x-ct)$, where $c$ is the propagation speed of a solitary wave.
Substituting this solution into \eqref{nl} and then
integrating twice and using the asymptotic boundary conditions, we have \begin{equation}
c^2 \phi=[\beta * (\phi+\phi^p)]. \label{iteras}
\end{equation}
If we use the Fourier transform,
\begin{equation}
u(x)=\frac{1}{2\pi} \int_{-\infty}^{\infty} \hat{u} (k)e^{ik x} dk, \hspace*{30pt} \hat{u}(k)=\int_{-\infty}^{\infty} u(x) e^{-ik x} dx,
\end{equation}
the eq. \eqref{iteras} can be written as
\begin{equation}
[c^2-\hat{\beta}(k)]\hat{\phi}
(k)= \hat{\beta} (k) \widehat{\phi^p}(k). \label{iteras2}
\end{equation}
A simple iterative algorithm for numerical calculation of $\hat{\phi}(k)$ for the above  eq. can be proposed in the form
\begin{equation}
\hat{\phi}_{n+1}(k)= \frac{\hat{\beta} (k)}{c^2-\hat{\beta}(k)}\widehat{\phi_n^p}(k). \label{algor}
\end{equation}
where $\hat{\phi}_n (k)$ is the Fourier transform of $\phi_n (x)$ which is the   $n^{th}$ iteration of the numerical solution.
Although there exists a fixed point $\hat{\phi}_n (k)$ of the eq. \eqref{iteras}, the algorithm \eqref{algor} diverges. To ensure the convergence, we
add a stabilizing factor $M_n$ as in \cite{petviashvili}. The new algorithm for the nonlocal nonlinear wave eq. can be proposed in the form
\begin{equation}
\hat{\phi}_{n+1}(k)= M_n^\gamma \frac{\hat{\beta} (k)}{c^2-\hat{\beta}(k)}\widehat{\phi^p_n}(k). \label{newalgor}
\end{equation}
where the stabilizing factor is
\begin{equation}
M_n=\frac{\int_{-\infty}^{\infty} (c^2-\hat{\beta}(k))~ [\hat{\phi}_n(k)]^2 dk }{\int_{-\infty}^{\infty} \hat{\beta} (k) \widehat{\phi^p_n}(k)\hat{\phi}_n(k) dk  } \label{stab}
\end{equation}
and $\gamma$ is a free parameter. Note that  we can only construct the solitary wave solutions for the nonlocal nonlinear wave eq. under the assumption
\begin{equation}
c^2-\hat{\beta}(k) \neq 0.  \label{condc}
\end{equation}
for all $k\in \mathbb{R}$.

\setcounter{equation}{0}
\section{Numerical implementation}

In this section we present some numerical results  to understand the effect of the kernel function $\beta(x)$ on the solution.
Throughout the section we consider the quadratic nonlinearity, i.e. $g(u)=u^2$.
In the first numerical experiment, we  consider the  general kernel function satisfying \eqref{condbeta} whose Fourier transform is
\begin{equation}
\hat{\beta}(k)=\frac{1}{1+k^2+\eta k^2 \sin(k^2)}. \label{kernelsin}
\end{equation}
Here $\eta$ is a positive parameter.  Note that $\hat{\beta}(k)$ is a decreasing function of $k$.
Therefore, the speed must  be chosen as  $c^2>1$  to satisfy the condition \eqref{condc}.
First, we construct a solitary wave solution corresponding to the above kernel for $\eta=1$
by using the  Petviashvili's iteration method. For this aim, we use the pseudo-spectral approximation for the eq. \eqref{newalgor}.
The initial guess is
\begin{equation}
\phi_0(x)=e^{-x^2}
\end{equation}
and we take the  the speed  as $c=1.08$. In the left panel of  Figure 1, we show  the solitary wave profile of the eq. \eqref{nl}  on the interval $[-100, 100]$ with $N=1024$. In the right panel of Figure 1, we  show the variation of the logarithm of the residual error with the  number of iterations. Here the residual error is defined as
\begin{equation}
\mbox{Residual error}= \|M \phi_n\|_\infty
\end{equation}
where
\begin{equation}
M\phi=c^2 \phi-[\beta * (\phi+\phi^p)].
\end{equation}
We observe that the fastest convergence occurs when $\gamma=2$.
\vspace*{-10pt}
\begin{figure}[h!bt]
\includegraphics[width=5.5in]{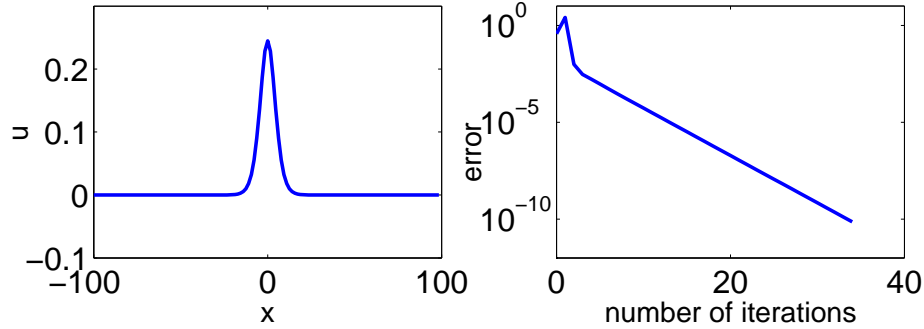}

\vspace*{-140pt}
\caption{ \small{The numerical solitary wave profile for the speed $c=1.08$ and the variation of the logarithm of the residual error with the  number of iterations.}}
\end{figure}

\vspace*{10pt}
\noindent
To investigate the time evolution of the wave, the initial condition $u(x,0)$ is chosen  as the profile constructed by the Petviashvili's  iteration  method (left panel of Figure 1).
On the other hand, we also need the initial condition $u_t(x,0)$.
The traveling wave solutions of the eq. \eqref{nl} are given in the form $u(x,t)=\phi(x-ct)$. Therefore,  the initial condition $u_t(x,0)$ is chosen as   $u_t(x,0)=-cu_x(x,0)$ and the $x$-derivative is evaluated by using discrete Fourier transform. In the left panel of  Figure 2, we illustrate the time evolution of the wave  corresponding to these initial conditions.
In the right panel of Figure 2, the variation of the change in the conserved quantity $E$ (energy)
with time is presented.

\vspace*{10pt}
\begin{figure}[h!bt]
%\hspace*{-50pt}
\begin{minipage}[t]{0.44\linewidth}
\hspace*{-55pt}
\includegraphics[width=3.6in]{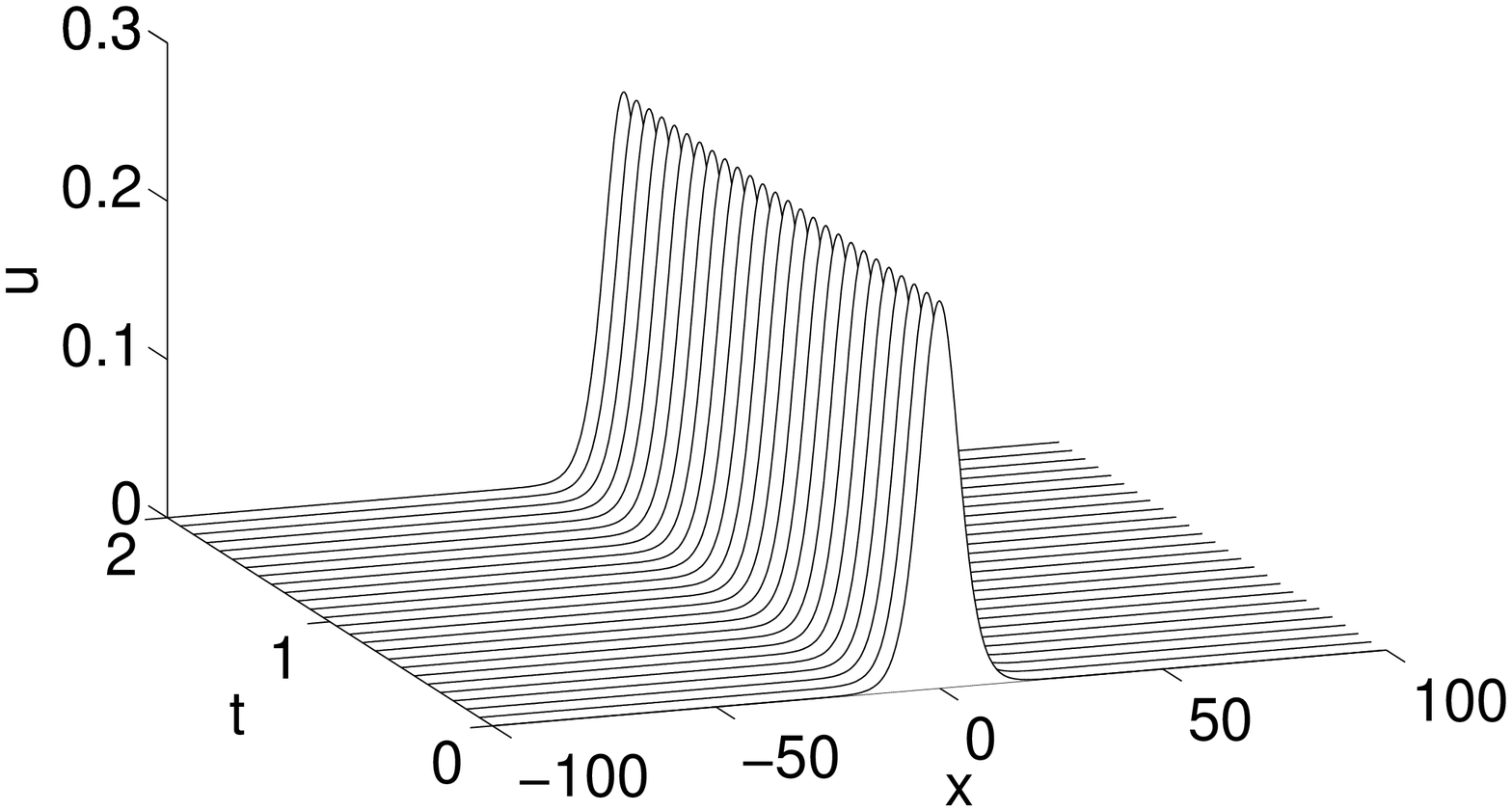}
\end{minipage}
%\hspace*{20pt}
\begin{minipage}[t]{0.44\linewidth}
\hspace*{15pt}
\includegraphics[width=3.2in]{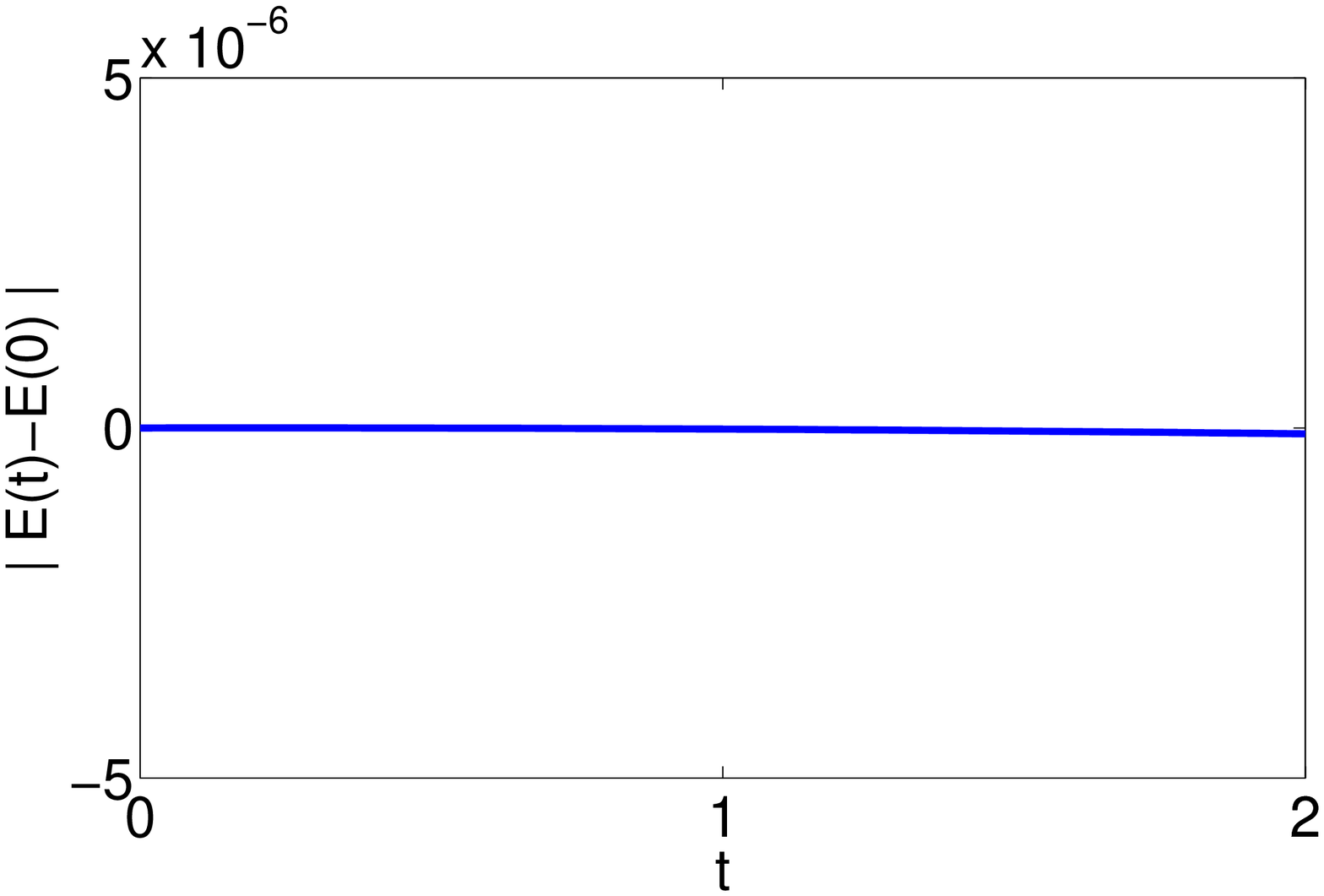}
\end{minipage}
\caption{\small{The time evolution of the wave generated by Petviashvili's  iteration  method (left panel) and the variation of the change in the conserved quantity $E$ (energy) with time (right panel).}}
\end{figure}

\noindent
We point out that the kernel \eqref{kernelsin} corresponds to the kernel giving rise to the IBq equation when $\eta=0$. In this case, if the nonlinearity is quadratic, the  relation between the amplitude $A$ and the speed parameter $c$ is given by $A=1.5(c^2-1)$. In Figure 3, the dashed line shows the  variation of the amplitude with the speed parameter for the IBq equation.
The solid lines show the variation of the amplitude
with the speed parameter corresponding to kernel \eqref{kernelsin} for $\eta=10$, $\eta=5$, $\eta=1$ and  $\eta=0.1$.

\vspace*{10pt}
\begin{figure}[h!bt]
\begin{center}
\includegraphics[scale=0.5]{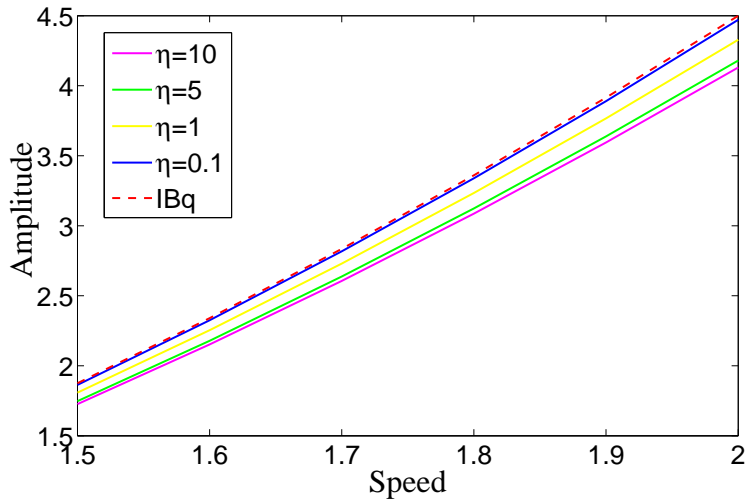}
\end{center}

\vspace*{-12pt}
\caption{ \small{Variation of the amplitude with the speed parameter for the IBq eq. and for
the nonlocal nonlinear wave eq. with the kernel  \eqref{kernelsin} for various values of $\eta$.}}
\end{figure}

\noindent
The question arises naturally how the solitary wave solutions of the nonlocal nonlinear wave equation
behaves when $\eta\rightarrow 0^+$. For this aim, we perform some numerical tests for various values of $\eta$
by using the initial condition
\begin{eqnarray}
&& u(x,0)=\frac{1}{4} {\mbox{sech}}^2(\frac{x}{\sqrt{28}}),   \label{ibq-initial1} \\
&& v(x,0)=\frac{1}{2} \,{\mbox{sech}}^2\left( \frac{x}{\sqrt{28}}\right)\,
            \tanh\left( \frac{x}{\sqrt{28}}\right).  \label{ibq-initial2}
\end{eqnarray}
The given initial condition corresponds to the exact solitary wave solution of the IBq equation. The problem is solved on the interval $-100\leq x \leq 100$ for times up to $T=10$.
In Figure $4$, the solid line shows the exact solution of the IBq equation given in \cite{borluk} which corresponds to the
initial conditions \eqref{ibq-initial1}-\eqref{ibq-initial2}.
The dashed lines show the numerical solution of the nonlocal nonlinear wave equation obtained by the kernel function whose Fourier
transform is given by \eqref{kernelsin} for  $\eta=10$, $\eta=5$, $\eta=1$ and  $\eta=0.1$ at the final time $T=10$. The numerical tests show that the solutions of the nonlocal nonlinear wave equation converge to the  solitary wave solution of the IBq equation as $\eta\rightarrow 0^+$.

\begin{figure}[h!bt]
\begin{minipage}[t]{0.44\linewidth}
\hspace*{-30pt}
\includegraphics[width=3in,height=1.8in]{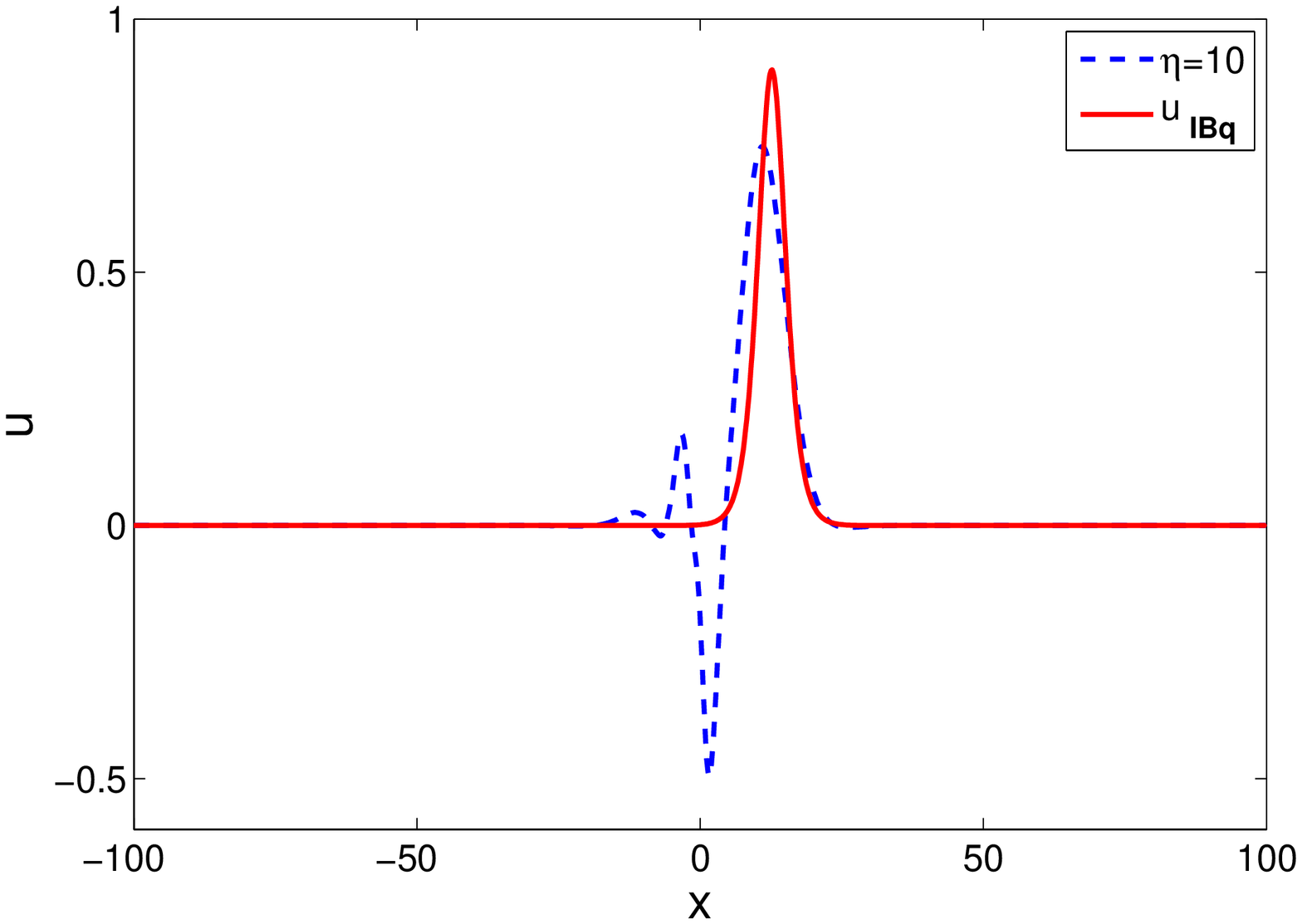}
\end{minipage}
\hspace*{30pt}
\begin{minipage}[t]{0.44\linewidth}
\hspace*{-5pt}
\includegraphics[width=3in,height=1.8in]{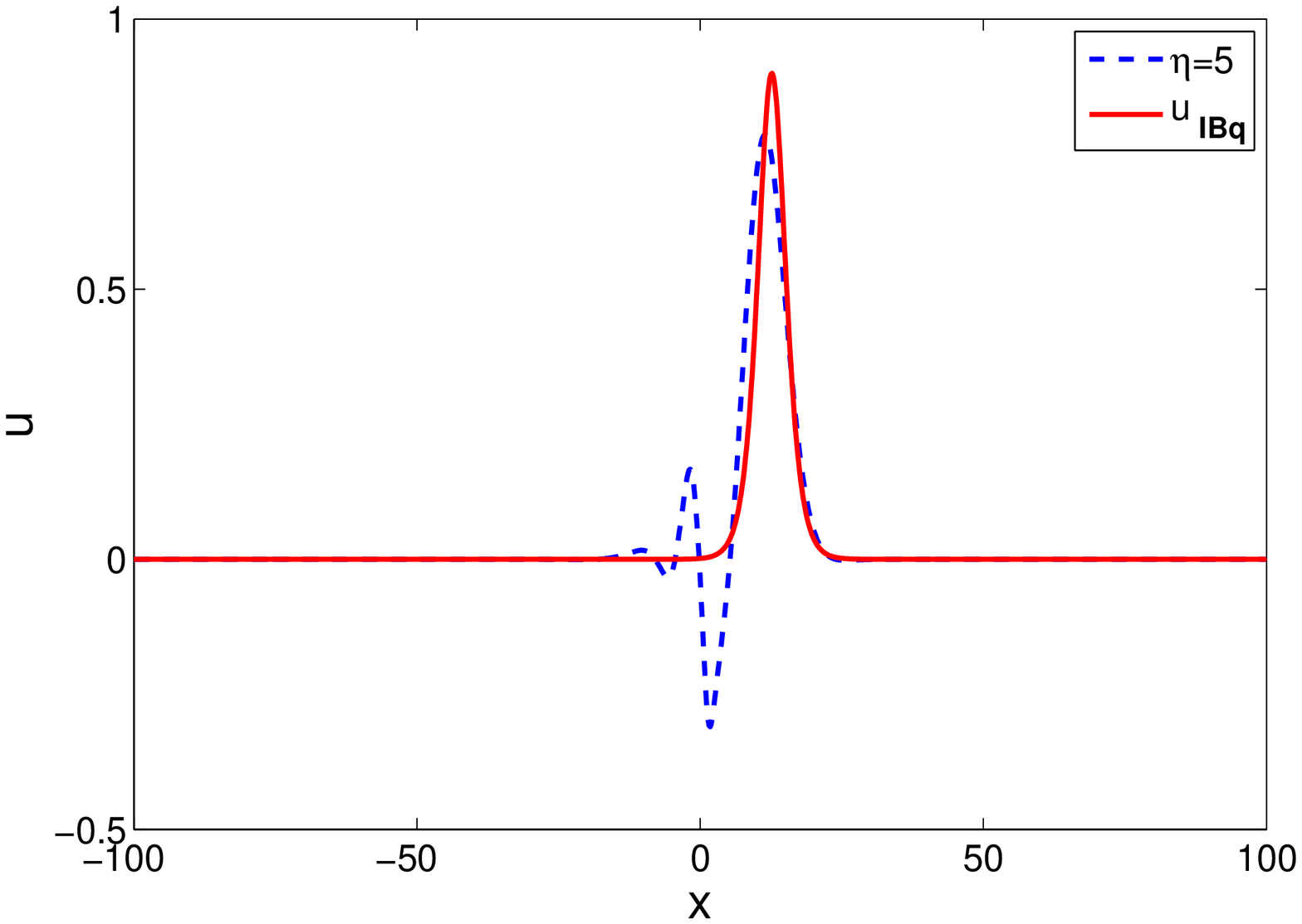}
\end{minipage}
\end{figure}
\begin{figure}[h!bt]

\vspace*{-5pt}
\begin{minipage}[t]{0.44\linewidth}
\hspace*{-30pt}
\includegraphics[width=3in,height=1.8in]{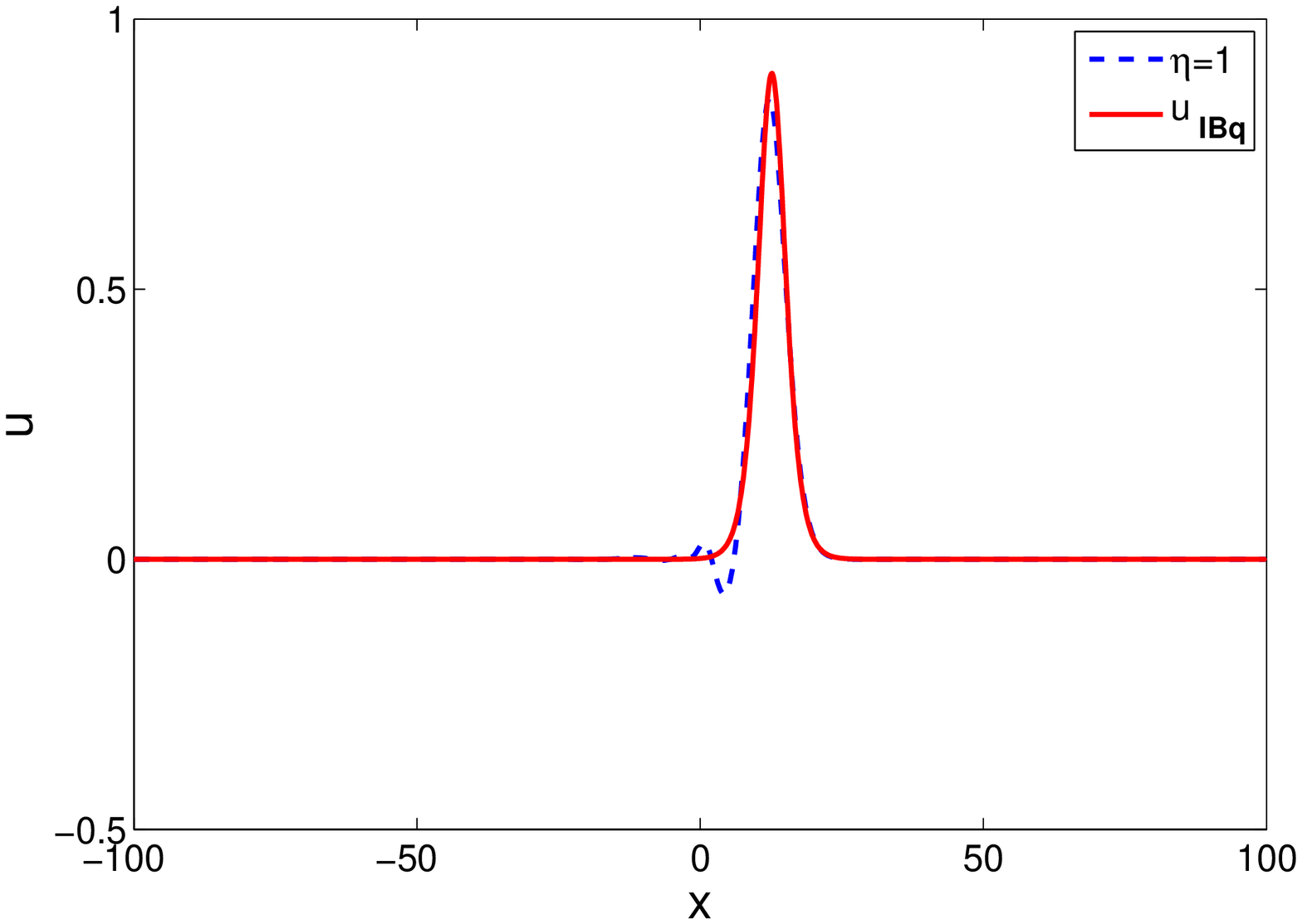}
\end{minipage}
\hspace*{30pt}
\begin{minipage}[t]{0.44\linewidth}
\hspace*{-5pt}
\includegraphics[width=3in,height=1.8in]{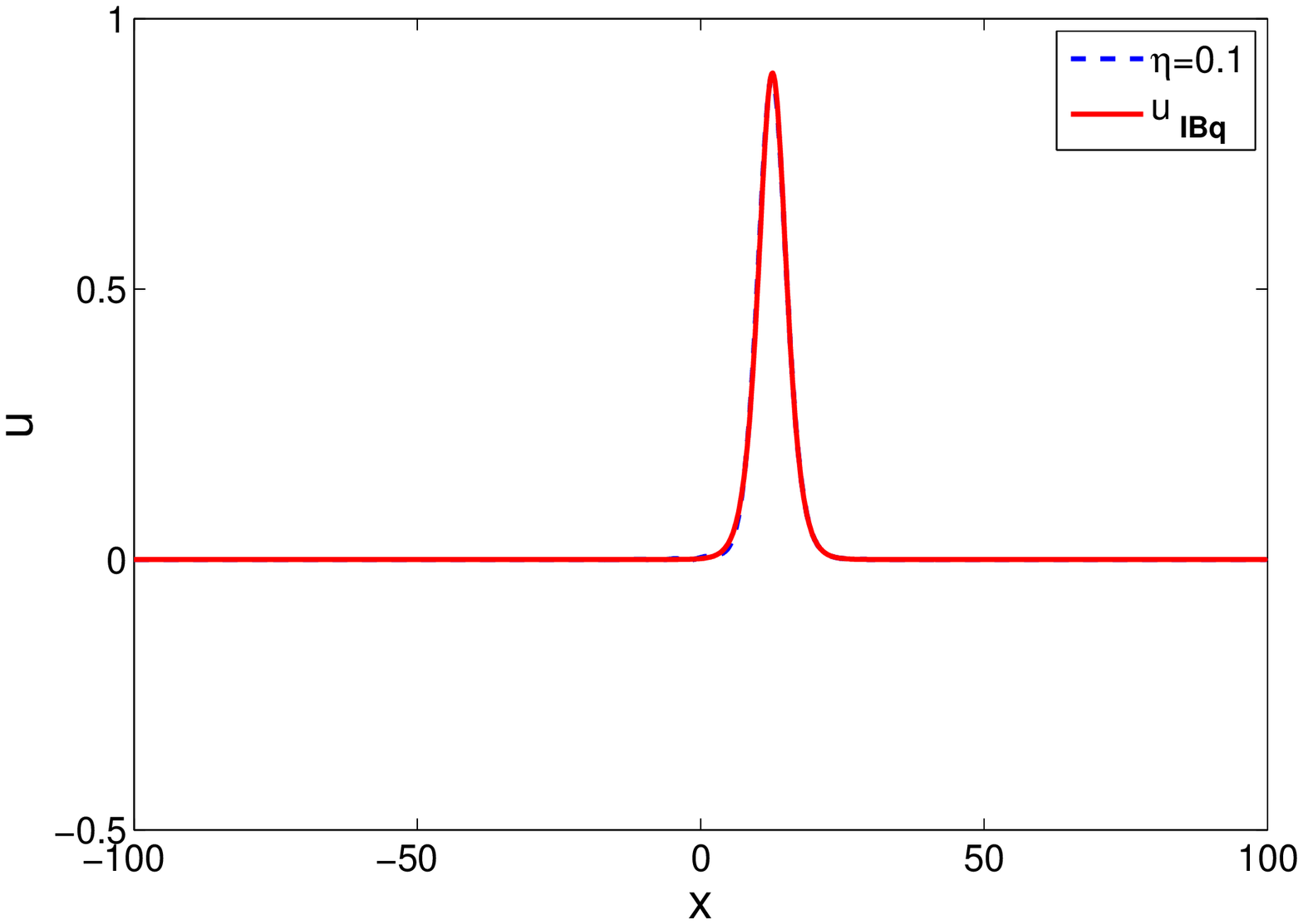}
\end{minipage}
\caption{\small{The exact solution of the IBq equation and the numerical solution of the nonlocal nonlinear wave equation
with \eqref{kernelsin}. }}
\end{figure}

\par
Next, we  consider the  kernel function  whose Fourier transform is
\begin{equation}
\hat{\beta}(k)=\frac{1}{1+k^2+ k^4}+ \frac{\mu}{1+k^4}. \label{kernelfrac}
\end{equation}
Here $\mu$ is a positive parameter. As $\hat{\beta}(k)$ is a decreasing function of $k$,
 the speed must  be chosen as  $c^2>1+\mu$  to satisfy the condition \eqref{condc}. This kernel corresponds to the kernel
giving rise to the HBq equation with $\eta_1=\eta_2=1$  when
$\mu=0$.   To observe the behavior of the solutions as $\mu\rightarrow 0^+$, we perform some numerical tests for various values of $\mu$ by using the initial condition
\begin{eqnarray}
&& u(x,0)=\frac{105}{266}\, {\mbox{sech}}^4(\frac{x}{2\sqrt{13}}), \label{hbq-initial1}  \\
&& v(x,0)= \frac{105}{133^{3/2}}\,{\mbox{sech}}^4\left(\frac{x}{2\sqrt{13}}~\right)\,
            \tanh\left( \frac{x}{2\sqrt{13}}~\right).  \label{hbq-initial2}
\end{eqnarray}
The given initial condition gives rise to the exact solitary wave solution of the HBq equation. In Figure 5, the solid line shows the exact solution of the HBq equation given in \cite{topkarci} which corresponds to the
initial conditions \eqref{hbq-initial1}-\eqref{hbq-initial2}.
The dashed lines show the numerical solution of the nonlocal nonlinear wave equation obtained by the kernel function whose Fourier
transform is given in \eqref{kernelfrac} for  $\mu=10$, $\mu=5$, $\mu=1$ and  $\mu=0.1$ at the final time $T=10$.
The numerical tests show that the solutions of the nonlocal nonlinear wave equation converge to the the solution of the HBq equation
as $\mu\rightarrow 0^+$.

\begin{figure}[h!bt]
\begin{minipage}[t]{0.44\linewidth}
\hspace*{-30pt}
\includegraphics[width=3in,height=1.8in]{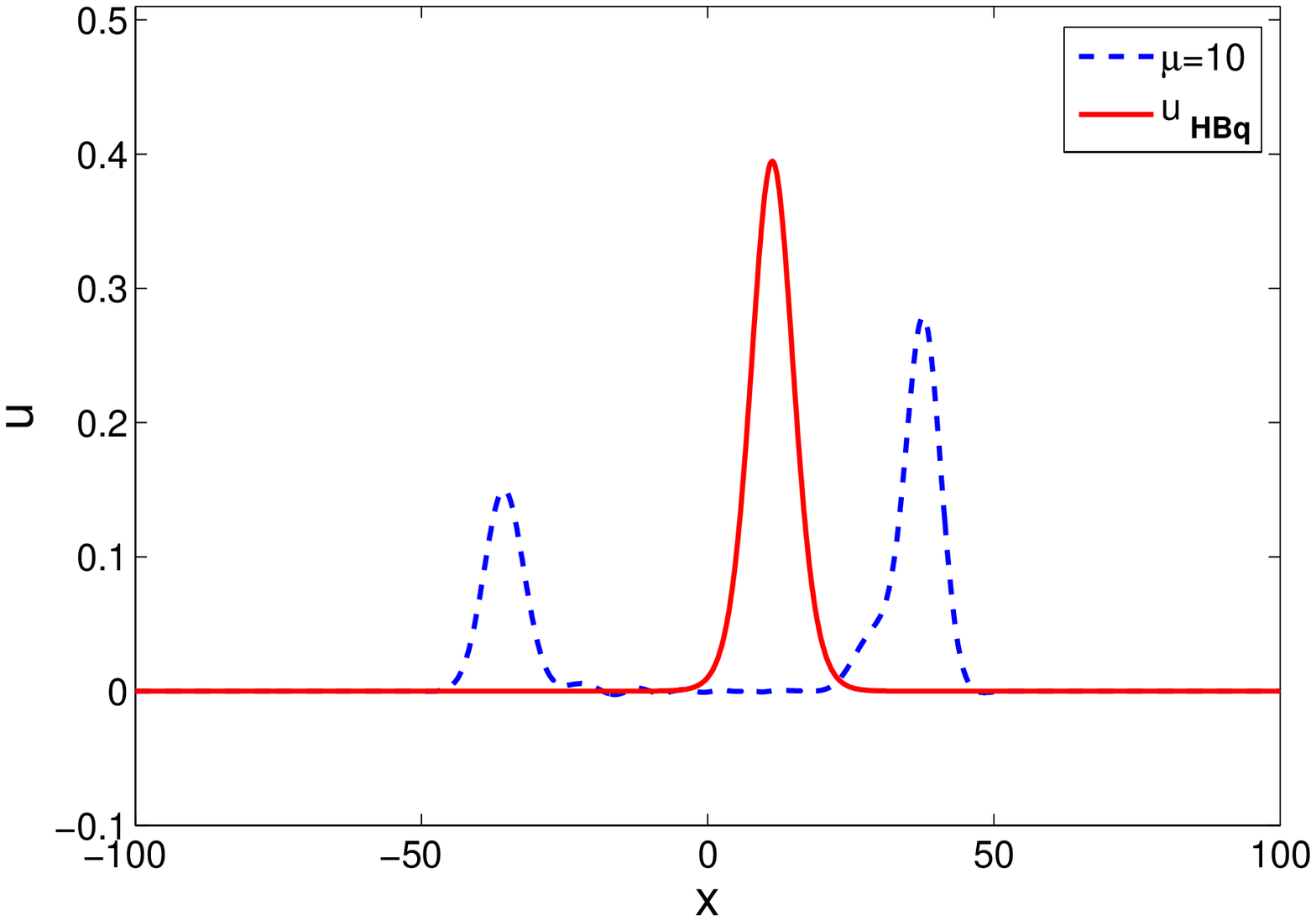}
\end{minipage}
\hspace*{30pt}
\begin{minipage}[t]{0.44\linewidth}
\hspace*{-5pt}
\includegraphics[width=3in,height=1.8in]{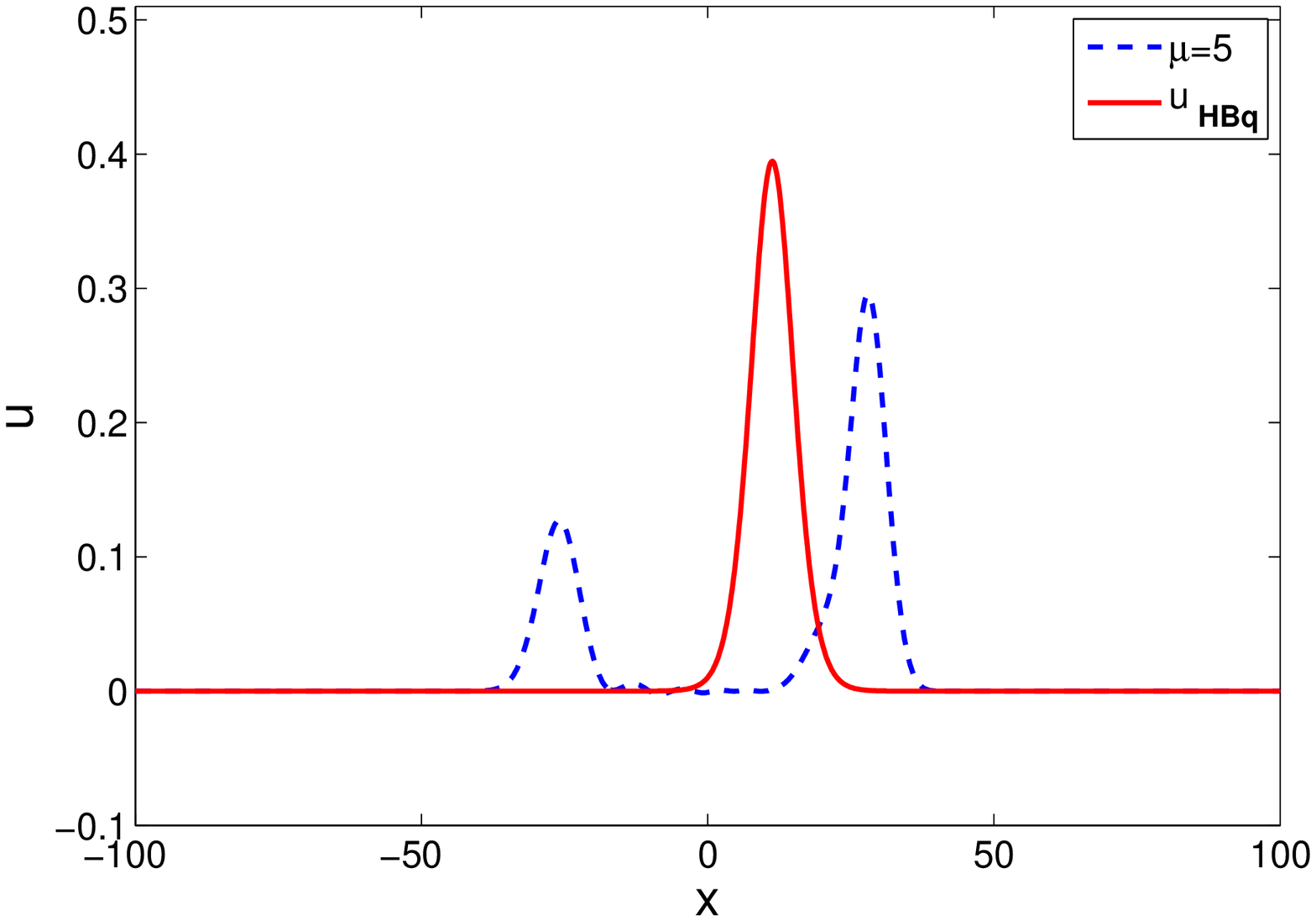}
\end{minipage}
\end{figure}
\begin{figure}[h!bt]

\vspace*{-15pt}
\begin{minipage}[t]{0.44\linewidth}
\hspace*{-30pt}
\includegraphics[width=3in,height=1.8in]{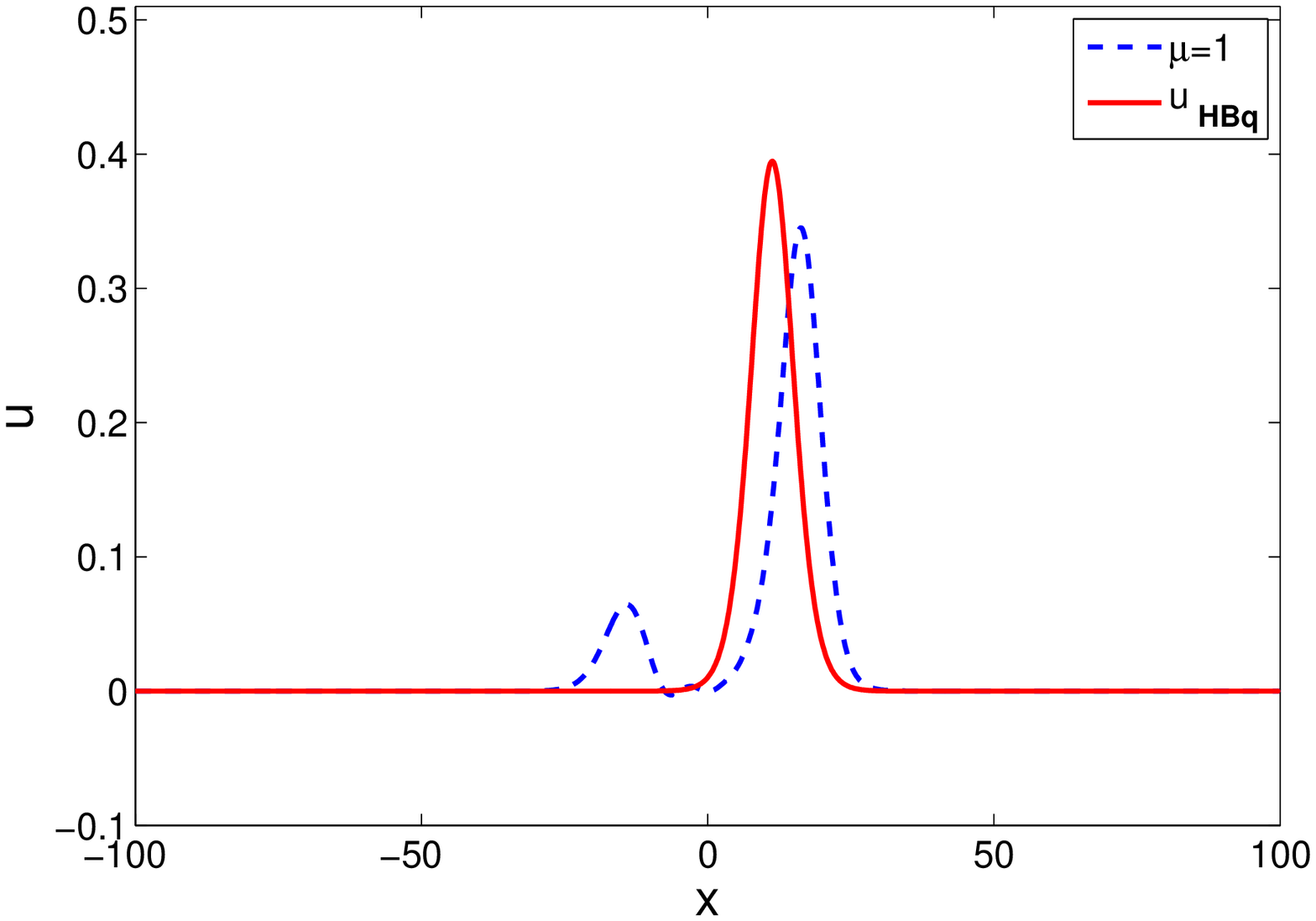}
\end{minipage}
\hspace*{30pt}
\begin{minipage}[t]{0.44\linewidth}
\hspace*{-5pt}
\includegraphics[width=3in,height=1.8in]{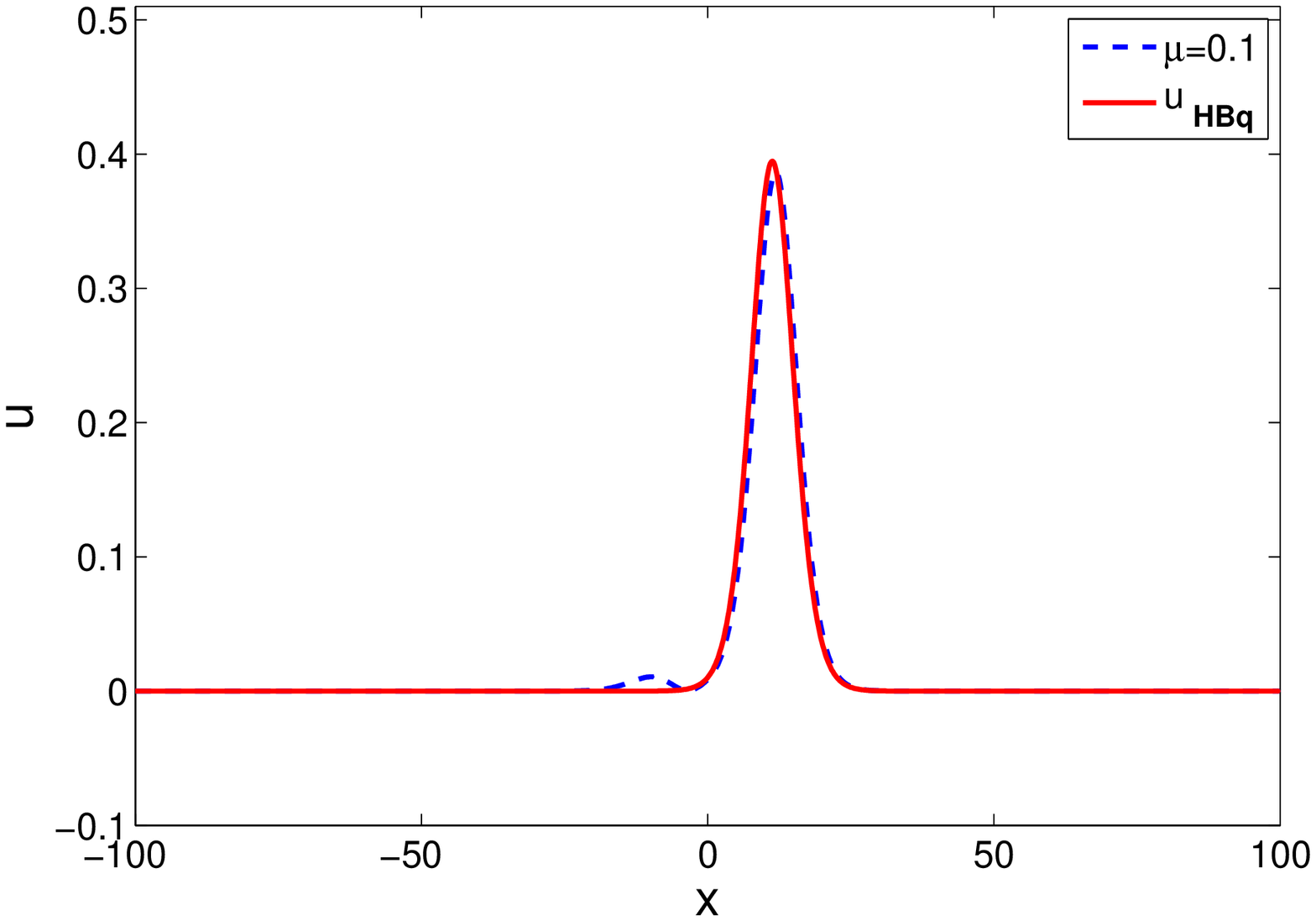}
\end{minipage}
\caption{\small{The exact solution of the HBq equation and the numerical solution of the nonlocal nonlinear wave equation
with \eqref{kernelfrac}.}}
\end{figure}
\par
For the last experiment, our aim is to investigate whether  the above convergence property is also valid for the blow-up solutions or not.
The blow-up theorem for the nonlocal nonlinear wave equation given  in \cite{duruk1} has been restated as follows:

\begin{theorem} (Theorems $3.4$ and $5.2$ of \cite{duruk1})
Let $s>1/2$ and $r\geq 2$. Then, there is some $T>0$ such that the Cauchy problem \eqref{nl}-\eqref{ic}
    is well-posed with solution in $C^2([0,T],H^s(\Omega))$ for initial data $\phi, \psi \in H^s(\Omega)$.
    Furthermore, suppose that $P\phi, P\psi \in L^{2}(\Omega)$ and $~~G(\phi )\in L^{1}(\Omega)$. If there is some
    $\nu >0$ such that
    \begin{equation*}
    p\,g\left( p\right) \leq 2\nu p^2+2\left( 1+2\nu \right) G\left( p\right) \mbox{ for all }p\in \mathbb{R},
    \end{equation*}
    and the initial energy
    \begin{equation*}
    E\left( 0\right) =\left\Vert P\psi \right\Vert^{2}+\left\Vert \phi \right\Vert^{2}+ 2\int_{\Bbb R} G\left( \phi\right) dx<0~,
    \end{equation*}
    then the solution $u$ of the Cauchy problem (\ref{nl})-(\ref{ic}) blows up in finite time.
\end{theorem}

\noindent
For the numerical tests,  we use the kernel whose Fourier transform is given  in \eqref{kernelsin} and we choose the initial conditions \begin{equation}
 \phi(x)=4\left(\frac{2x^{2}}{3}-1\right) e^{-\frac{~x^{2}}{3}} , \hspace*{30pt} \psi(x)=\left(x^{2}-1\right)e^{-\frac{~x^{2}}{3}}
 \label{inamp}
\end{equation}
as in \cite{godefroy}. Thus, the conditions for the blow-up  theorem are satisfied for $\nu=\frac{1}{4}$.
The problem is solved on the interval $-10 \leq x \leq 10$ for times up to $T=3$.
In Figure 6, we present the variation of the $L_\infty$ norm of the numerical solution obtained by using the Fourier pseudo-spectral scheme. The numerical results strongly indicate that the  blow-up time of the nonlocal nonlinear wave equation converges
to $t^* \approx 1.8$,  which is the blow-up time of the IBq equation, as $\eta\rightarrow 0^+$.

%\vspace*{-20pt}
\begin{figure}[h!bt]
\centering
\includegraphics[scale=0.46]{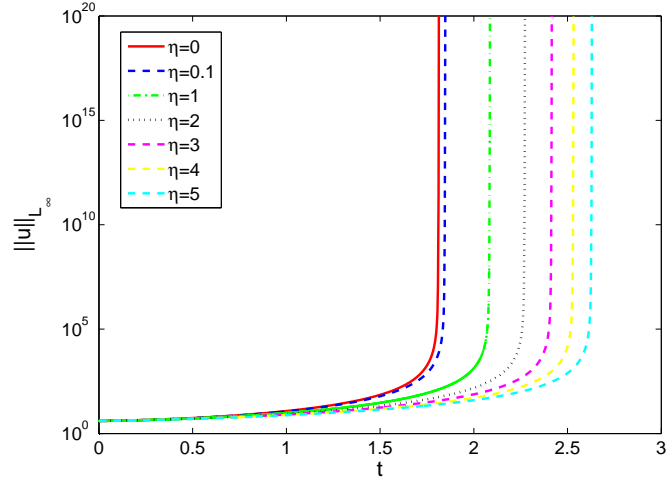}
\caption{ \small{$L_\infty$ norm of the numerical solution for the increasing time. }}
\end{figure}
\newpage
\vspace*{10pt}
\noindent \textbf{Acknowledgement:} This work has been supported by the Scientific and Technological Research Council of Turkey
(TUBITAK) under the project MFAG-113F114. The authors  gratefully  acknowledge
  to the anonymous reviewers for the constructive comments and valuable suggestions which improved the original paper.

%\bibliographystyle{elsart-num-sort}
%\bibliography{kaynak}

\end{document}